\documentclass[12pt,reqno]{amsart}
\usepackage{amsthm}
\usepackage{amssymb}
\usepackage{graphics}
\usepackage{tikz}
\usetikzlibrary{shapes,backgrounds,calc}
\usepackage{latexsym}
\usepackage{multicol}
\usepackage{verbatim,enumerate}

\usepackage{cite}
\usepackage{array}
\usepackage[colorlinks=true, linkcolor=blue, citecolor=blue, urlcolor=black]{hyperref}
\usepackage{hyperref}
\usepackage{amsmath, amscd,url}
\usepackage{multirow}
\usepackage{bigstrut}
\usepackage{longtable}

\usepackage{stackengine}
\newcommand\xrowht[2][0]{\addstackgap[.5\dimexpr#2\relax]{\vphantom{#1}}}

\advance\textwidth by 1.3in \advance\oddsidemargin by -.6in \advance\evensidemargin by -.6in
\theoremstyle{definition}

\newtheorem{theorem}{Theorem}[section]
\newtheorem{lemma}[theorem]{Lemma}
\newtheorem{corollary}[theorem]{Corollary}
\theoremstyle{definition}
\newtheorem{definition}[theorem]{Definition}
\newtheorem{example}[theorem]{Example}
\newtheorem{proposition}[theorem]{Proposition}

\theoremstyle{remark}

\newcommand{\F}{\mbox{$\mathbb F$}}
\theoremstyle{definition}

\newcounter{cnt}
 \makeatletter
\def\mydggeometry{\makeatletter\dg@YGRID=1\dg@XGRID=20\unitlength=0.003pt\makeatother}
\makeatother \theoremstyle{remark}

\numberwithin{equation}{section}
\let\bwdg\bigwedge
\def\bigwedge{{\textstyle\bwdg}}

\newcommand{\nc}{\newcommand}
\newcommand{\rnc}{\renewcommand}

\nc{\cal}{\mathcal} \nc{\goth}{\mathfrak} \rnc{\bold}{\mathbf}

\nc\bomega{{\mbox{\boldmath $\omega$}}} \nc\bpsi{{\mbox{\boldmath $\Psi$}}}
 \nc\balpha{{\mbox{\boldmath $\alpha$}}}
 \nc\bpi{{\mbox{\boldmath $\pi$}}}
 \nc\bvpi{{\mbox{\boldmath $\varpi$}}}
\nc\chara{\operatorname{ch}}

  \nc\bxi{{\mbox{\boldmath $\xi$}}}
\nc\bmu{{\mbox{\boldmath $\mu$}}} \nc\bcN{{\mbox{\boldmath $\cal{N}$}}} \nc\bcm{{\mbox{\boldmath $\cal{M}$}}} \nc\blambda{{\mbox{\boldmath
$\lambda$}}}\nc\bnu{{\mbox{\boldmath $\nu$}}}

\makeatletter
\def\section{\def\@secnumfont{\mdseries}\@startsection{section}{1}%
  \z@{.7\linespacing\@plus\linespacing}{.5\linespacing}%
  {\normalfont\scshape\centering}}
\def\subsection{\def\@secnumfont{\bfseries}\@startsection{subsection}{2}%
  {\parindent}{.5\linespacing\@plus.7\linespacing}{-.5em}%
  {\normalfont\bfseries}}
\makeatother

 \nc{\Hom}{\operatorname{Hom}}
  \nc{\mode}{\operatorname{mod}}
\nc{\End}{\operatorname{End}} \nc{\wh}[1]{\widehat{#1}} \nc{\Ext}{\operatorname{Ext}} \nc{\ch}{\text{ch}} \nc{\ev}{\operatorname{ev}}
\nc{\Ob}{\operatorname{Ob}} \nc{\soc}{\operatorname{soc}} \nc{\rad}{\operatorname{rad}} \nc{\head}{\operatorname{head}}

 \nc{\Cal}{\cal} \nc{\Xp}[1]{X^+(#1)} \nc{\Xm}[1]{X^-(#1)}
\nc{\on}{\operatorname} \nc{\Z}{{\bold Z}} \nc{\J}{{\cal J}}  \nc{\Q}{{\bold Q}}

\nc{\N}{{\bold N}}  \nc\boa{\bold a} \nc\bob{\bold b} \nc\boc{\bold c} \nc\bod{\bold d} \nc\boe{\bold e} \nc\bof{\bold f} \nc\bog{\bold g}
\nc\boh{\bold h} \nc\boi{\bold i} \nc\boj{\bold j} \nc\bok{\bold k} \nc\bol{\bold l} \nc\bom{\bold m} \nc\bon{\mathbb n} \nc\boo{\bold o}
\nc\bop{\bold p} \nc\boq{\bold q} \nc\bor{\bold r} \nc\bos{\bold s} \nc\boT{\bold t} \nc\boF{\bold F} \nc\bou{\bold u} \nc\bov{\bold v}
\nc\bow{\bold w} \nc\boz{\bold z}\nc\ba{\bold A} \nc\bb{\bold B} \nc\bc{\mathbb C} \nc\bd{\bold D} \nc\be{\bold E} \nc\bg{\bold
G} \nc\bh{\bold H} \nc\bi{\bold I} \nc\bj{\bold J} \nc\bk{\bold K} \nc\bl{\bold L} \nc\bm{\bold M} \nc\bn{\mathbb N} \nc\bo{\bold O} \nc\bp{\bold
P} \nc\bq{\bold Q} \nc\br{\bold R} \nc\bs{\bold S} \nc\bt{\bold T} \nc\bu{\bold U} \nc\bv{\bold V} \nc\bw{\bold W} \nc\bz{\mathbb Z} \nc\bx{\bold
x} \nc\KR{\bold{KR}} \nc\rk{\bold{rk}} \nc\het{\text{ht }}

\nc\toa{\tilde a} \nc\tob{\tilde b} \nc\toc{\tilde c} \nc\tod{\tilde d} \nc\toe{\tilde e} \nc\tof{\tilde f} \nc\tog{\tilde g} \nc\toh{\tilde h}
\nc\toi{\tilde i} \nc\toj{\tilde j} \nc\tok{\tilde k} \nc\tol{\tilde l} \nc\tom{\tilde m} \nc\ton{\tilde n} \nc\too{\tilde o} \nc\toq{\tilde q}
\nc\tor{\tilde r} \nc\tos{\tilde s} \nc\toT{\tilde t} \nc\tou{\tilde u} \nc\tov{\tilde v} \nc\tow{\tilde w} \nc\toz{\tilde z} \nc\woi{w_{\omega_i}}

\begin{document}
\setcounter{section}{0}
\setcounter{tocdepth}{1}


\title{discriminant and integral basis of pure nonic fields}

\author{Anuj Jakhar}
\address{Indian Institute of Technology (IIT) Madras, India.}
\email{anujjakhar@iitm.ac.in, anujiisermohali@gmail.com}
\author{Neeraj Sangwan}
\address{LNMIIT, Jaipur, Rajasthan.}
\email{neeraj@lnmiit.ac.in, ms09089@gmail.com}


\subjclass [2010]{11R04, 11R29}
\keywords{Discriminant, Integral basis, Monogenity}

\begin{abstract}
Let  $K = \Q(\theta)$ be an algebraic number field with $\theta$ satisfying an irreducible polynomial $x^{9} - a$ over the field $\Q$ of rationals and $\Z_K$ denote the ring of algebraic integers of $K$. In this article, we provide the exact power of each prime which divides the index of the subgroup $\Z[\theta]$ in $\Z_K$. Further, we give a $p$-integral basis of $K$ for each prime $p$. These $p$-integral bases lead to a construction of an integral basis of $K$ which is illustrated with examples. 
\end{abstract}

\maketitle
\vspace*{-0.3in}
\section{Introduction}\label{intro}
Let $\Z_K$ denote the ring of algebraic integers of an algebraic number field $K$. Let  degree of $K$ equal to $n$ over the field $\Q$ of rational numbers. The problem of computation of a $\Z$-basis of $\Z_K$ (called an integral basis of $K$) is in general a difficult task and has attracted the attention of many mathematicians (cf. \cite{ANH2}, \cite{Ded}, \cite{Fun}, \cite{Ga-Re}, \cite{HN}, \cite{Jakhar}, \cite{Jak}, \cite{JS}, \cite{JKS}, \cite{Wes}). In 1900, Dedekind \cite{Ded} gave an explicit integral basis for pure cubic fields. In 1910, Westlund \cite{Wes} described an integral basis for pure prime degree number fields. In 1984, Funakura \cite{Fun} provided an integral basis of all pure quartic fields. In 2015, Hameed and Nakahara \cite{HN} determined an integral basis of those pure octic fields $\Q(\sqrt[8]{a})$, where $a$ is squarefree integer. In 2017, Ga{\'a}l and Remete \cite{Ga-Re} gave an integral basis of $\Q(\sqrt[n]{a})$, where the integer $a$ is squarefree and $3\leq n\leq 9.$ 
Recently, Jakhar \cite{Jakhar}  provided an explicit integral basis of number fields of the type $\Q(\sqrt[p_1p_2]{a})$, where $p_1, p_2$ are distinct prime numbers.

For a prime number $p$, let $\Z_{(p)}$ denote the localization of the ring $\Z$ of integers at the prime ideal $p\Z$ and $S_{(p)}$ denote the integral closure of $\Z_{(p)}$ in an algebraic number field $K$. The ring $S_{(p)} = \{\frac{\alpha}{b} | \alpha \in \Z_K, b \in \Z\setminus p\Z\}$ is a free module over the principal ideal domain $\Z_{(p)}$ having rank equal to the degree of the extension $K/\Q$. A $\Z_{(p)}$-basis of the module $S_{(p)}$ is called a $p$-integral basis of $K$. Clearly an integral basis of an algebraic number field is its $p$-integral basis for each prime $p$.

Let $K = \Q(\theta)$ with $\theta$  a root of an irreducible polynomial $x^{9}-a \in \Z[x]$ over rationals. The aim of this paper is to provide the exact power of any  prime $p$ dividing the index $[\Z_K : \Z[\theta]]$ and to give a $p$-integral basis. These $p$-integral bases lead to a construction of an explicit integral basis of $K$ taking into considerations all the primes $p$ dividing the index of the subgroup $\Z[\theta]$ in $\Z_K$. We wish to point out here that if a prime $p$ does not divide $[\Z_K : \Z[\theta]]$, then it is easy to see that $\{1, \theta, \cdots, \theta^{8}\}$ is a $p$-integral basis of $K$.

Throughout the paper, for an algebraic number field $K = \Q(\theta)$ with $\theta \in \Z_K$, we shall denote the index of the subgroup $\Z[\theta]$ in $\Z_K$ by $I_K(\theta).$

The outline of the paper is as follows. In Section 2, we state our main result (Theorem \ref{main}) and provide some examples which illustrate the computation of an integral basis. In Section 3, we write some preliminary results for the proof of Theorem \ref{main}. In Section 4, we apply Section 3 and  compute $p$-integral basis of $\Q(\sqrt[9]{a})$ for each prime $p$ dividing $3a$. In Section 5, we write a proof of Theorem \ref{main}. 

\section{Main result and Examples}
In what follows, let $K = \Q(\theta)$ with $\theta$ a root of an irreducible polynomial $x^{9} - a$ belonging to $\Z[x]$, where  $a$ is a $9$-th power-free integer, i.e., $a$ is not divisible by the $9$-th power of any prime number.   Whenever  $3$ divides both $a$ and $v_{3}(a)$, we denote $\frac{v_{3}(a)}{3}$ by $c$; in this situation we can write $a = 3^{3c}b$ with $3 \nmid b$.

With the above notations, the following result provides the exact power of each prime $p$ dividing $I_K(\theta)$ and gives an explicit $p$-integral basis for all number fields of the type $\Q(\sqrt[9]{a})$.
\begin{theorem}\label{main}
Let $K = \Q(\theta)$ with $\theta$ a root of an irreducible polynomial $f(x) = x^{9}-a$, where $a$ is a $9$-th power-free integer. Let $b, c$ be as above. Then the following hold.
\begin{itemize}
\item[(1)] If $q$ is any prime divisor of $a$ such that $q$ does not divide $\gcd(9, v_q(a))$, then $v_q(I_K(\theta)) = \frac{8(v_q(a)-1)+\gcd(9, v_q(a)) - 1}{2}$ and the set 
$$S = \bigg\{\dfrac{\theta^i}{q^{\big\lfloor\frac{iv_q(a)}{9}\big\rfloor}} ~|~0\leq i\leq 8 \bigg\}$$ is a $q$-integral basis of $K$.
\item[(2)] If $3$ divides both  $a$, $v_{3}(a)$ and $k$ is the smallest positive integer such that $b\equiv k\mod 9$, then 
the highest power of $3$ dividing the index $I_K(\theta)$ is given by 
  $$v_{3}(I_K(\theta)) =\left\{\begin{array}{cl}
 12c,& \mbox{if}~ k\in \{4,5\};\\
{12c+1}, & \mbox{}~\mbox{if}~ k\in \{1,2,7,8\}.
\end {array}\right.$$
Moreover, for $k\in \{1,8\}$, let $q_j(\theta)$ denote the polynomial $\sum\limits_{s=0}^{3-j}\binom{3}{s}(3^{c}b)^{s}(\theta^{3}-3^{c}b)^{3-j-s}$, $1\leq j\leq 2$. Then the following set $S$ is a $3$-integral basis of $K$, where $q_1(\theta)$, $q_2(\theta)$ for $k\in \{2,4,5,7\}$ and $m_i$'s for $1\leq i\leq 7$ are as in Table \ref{tab:1} and Table \ref{tab:2} depending on the values of $c$ and $k$.
\begin{align*}
 S &= \{ 1,\theta, \frac{\theta^2}{3^{m_1}}, \frac{q_2(\theta)}{3^{m_2}}, \frac{\theta q_2(\theta)}{3^{m_3}}, \frac{\theta^2 q_2(\theta)}{3^{m_4}},\frac{q_1(\theta)}{3^{m_5}}, \frac{\theta q_1(\theta)}{3^{m_6}}, \frac{\theta^2 q_1(\theta)}{3^{m_7}}   \}.
\end{align*}
\begin{center}
\begin{longtable}[h!]{|m{0.7cm}|m{3.2cm}|m{3.75cm}|m{5.75cm}|} 
\hline 
  $c$ & $\phi(\theta)$   & $q_2(\theta)$ & $q_1(\theta)$\\ \hline
  &   &   &  \\
 1 & $\theta^3 - (-1)^{k}3b\theta - 3b$  & $\phi(\theta) + (-1)^k 9b\theta +9b$ &  $\phi(\theta)^2+ \phi(\theta) ((-1)^k 9b\theta +9b)$ $+ 27b^2 \theta^2 + (-1)^k 54 b^2\theta +(-1)^k 27b^3$ $+27b^2$  \bigstrut \xrowht{18pt} \\ \cline{1-4}
  &   &   &  \\
  2 & $\theta^3 - (-1)^{\lfloor \frac{k}{2}\rfloor}3b\theta^2 - 9b$  &$\phi(\theta) + (-1)^{\lfloor \frac{k}{2}\rfloor} 9  b  \theta^2 + 27  b ^2  \theta + (-1)^{\lfloor \frac{k}{2}\rfloor}108  b ^3$ $  + 27  b$ & $\phi(\theta)^2+ \phi(\theta) (-1)^{\lfloor \frac{k}{2}\rfloor} 9  b  \theta^2 + 27  b ^2  \theta + (-1)^{\lfloor \frac{k}{2}\rfloor}108  b ^3   + 27  b) + 405  b ^4  \theta^2  +(-1)^{\lfloor \frac{k}{2}\rfloor} 162  b ^2  \theta^2  + (-1)^{\lfloor \frac{k}{2}\rfloor} 243  b ^5  \theta + 486  b ^3  \theta + 729  b ^6   + (-1)^{\lfloor \frac{k}{2}\rfloor} 1944  b ^4  + 243  b ^2 $ \bigstrut \xrowht{18pt}\\ \hline
  \caption{\label{tab:1}The values of $q_1(\theta)$ and $q_2(\theta)$ for $k \in \{2,4,5,7\}$.}
\end{longtable}
\end{center}
\begin{center}
\begin{longtable}[h!]{|m{1.3cm}|m{0.7cm}|m{0.7cm}|m{0.7cm}|m{0.7cm}|m{0.7cm}|m{0.7cm}|m{0.7cm}|m{0.7cm}|} 
\hline 
  $k$ & $c$  &$m_1$ & $m_2$ & $m_3$ & $m_4$ & $m_5$ & $m_6$ & $m_7$\\ \hline
 $1,2,7, 8$&$1$ & $0$ & $1$ & $1$ & $2$ & $3$ & $3$ & $3$ \\ \hline
 $1,2,7, 8$&$2$ & $1$ & $2$ & $3$ & $3$ & $5$ & $5$ & $6$ \\ \hline
 $4,5$&$1$ & $0$ & $1$ & $1$ & $2$ & $2$ & $3$ & $3$ \\ \hline
 $4,5$&$2$ & $1$ & $2$ & $3$ & $3$ & $4$ & $5$ & $6$ \\ \hline
 
 \caption{\label{tab:2}The values of $m_i$'s for $1\leq i\leq 7$ depending on $c$ and $k$.}
\end{longtable}
\end{center}
\item[(3)] If a prime $q$ does not divide $3a$, then  $v_q(I_K(\theta)) = 0$ and the set $S = \{1, \theta, \cdots, \theta^{8}\}$ is a $q$-integral basis of $K$.
\item[(4)] If $3$ does not divide $a$, then 
\begin{itemize}
\item[(i)] When $9 \nmid (a^{2}-1)$, then $v_{3}(I_K(\theta)) = 0$ and the set  $\{1, \theta, \cdots, \theta^{8}\}$ is a $3$-integral basis of $K$.
\item[(ii)] When $v_3(a^{2}-1) = 2$, then $v_{3}(I_K(\theta)) = 3$ and the set $\{1, \theta, \theta^2, \cdots, \theta^5, \frac{\theta^6+a\theta^3+1}{3},$ $  \frac{(\theta^6+a\theta^3+1)\theta}{3},  \frac{(\theta^6+a\theta^3+1)\theta^2}{3}\}$ is a $3$-integral basis of $K$.
\item[(iii)] When $27 \mid (a^{2}-1)$, then $v_{3}(I_K(\theta)) = 4$ and the set $\{1, \theta, \theta^2, \cdots, \theta^5, \frac{\theta^6+a\theta^3+1}{3},$ $  \frac{(\theta^6+a\theta^3+1)\theta}{3},  \frac{\theta^8+\sum\limits_{j=0}^{7}(a\theta)^{j}}{9}\}$ is a $3$-integral basis of $K$.
\end{itemize}
\end{itemize}
\end{theorem}

\subsection{Construction of an explicit integral basis from $p$-integral basis}

If $L=\Q(\eta)$ is an algebraic number field of degree $n$ with $\eta$ an algebraic integer, then as is well known that there exists an integral basis $\mathcal B:=\{\beta_0,\cdots ,\beta_{n-1}\}$ of $L$ such that $\beta_0=1$, $$\beta_i=\frac{a_{i,0}+a_{i,1}\eta+\cdots +a_{i,i-1}\eta^{i-1}+\eta^i}{d_i}$$ with $a_{i,j},~d_i\in \Z$ and the positive integer $d_i$ dividing $d_{i+1}$ for $1\leq i\leq n-1$; moreover the numbers $d_i$ are uniquely determined by $\eta$ and the index  $[A_L : \Z[\eta]]$ to be denoted by $I_L(\eta)$ equals $\displaystyle\prod_{i=1}^{n-1}d_i$  (cf. \cite[Chapter 2, Theorem 13]{Mar}). Fix a prime $p$ and let $l_i$ denote its highest power dividing $d_i$, then $v_p(I_L(\eta))=l_1+l_2+\cdots +l_{n-1}.$ Since $\mathcal B$ is a $p$-integral basis of $L$, so is $\mathcal B^*:=\{1,\frac{\beta_1d_1}{p^{l_1}},\cdots ,\frac{\beta_{n-1}d_{n-1}}{p^{l_{n-1}}}\}$. It can be easily seen that if $\mathcal C=\{1,\gamma_1,\cdots ,\gamma_{n-1}\}$ is another $p$-integral basis of $L$ where $\gamma_i$'$s$ are of the type $$\gamma_i=\frac{c_{i,0}+c_{i,1}\eta+\cdots +c_{i,i-1}\eta^{i-1}+\eta^i}{p^{k_i}} $$ with $c_{i,j},~k_i$ in $ \Z$ for $1\leq i\leq n-1$, then on writing each member of $\mathcal B^*$ as a $\Z_{(p)}$-linear combination of members of $\mathcal C$ and vice versa, we see that $l_i=k_i~ \forall~ i$ and hence 
\begin{equation}\label{eqn p2}
v_p(I_L(\eta))=l_1+l_2+\cdots +l_{n-1}=k_1+k_2+\cdots +k_{n-1}.
\end{equation} In fact such a $p$-integral basis of an algebraic number field $L$ is easier to construct than constructing an integral basis of $L$ and these $p$-integral bases of $L$ with $p$ running over all primes dividing $I_L(\eta)$, quickly lead to a construction of an integral basis of $L$ as follows.

  Let $q_1, q_2, \cdots, q_r$ be all the distinct primes dividing the index $I_L(\eta).$ Let $\mathcal{B}_i = \{\alpha_{i0}, \alpha_{i1},$ $\cdots, \alpha_{i(n-1)}\}$ be a $q_i$-integral basis of $L$, $1\leq i\leq r$ with $$\alpha_{i0} = 1,~ \alpha_{ij} = \frac{\eta^j + \sum\limits_{s=0}^{j-1}c_{js}^{(i)}\eta^{s}}{q_i^{k_{i,j}}}, ~~1\leq j\leq n-1,$$ where $c_{js}^{(i)}$ and $0 \leq k_{i,j} \leq k_{i, j+1}$ are integers. Denote $\delta_{ij}$ belonging to $\Z[\eta]$ by $\delta_{ij} = \sum\limits_{s=0}^{j-1}c_{js}^{(i)}\eta^{s}.$ For a fixed $j$, $1\leq j\leq n-1$, let $z_{ij} = \prod\limits_{\ell=1,\ell\neq i}^{r}q_{\ell}^{k_{\ell,j}}$ for $1\leq i\leq r$. Since $\gcd\{z_{ij} | 1\leq i\leq r\} = 1$, there exist integers $u_{ij}$ such that $\sum\limits_{i=1}^{r}u_{ij}z_{ij} = 1$. Let $\beta_j$ equal $\sum\limits_{i=1}^{r}u_{ij}z_{ij}\delta_{ij}$.

With the above notations, we state the following theorem, which is proved in \cite[Theorem 2.2]{Jakhar}. This will be used to construct an explicit integral basis.
\begin{theorem}\label{main2}
Let $L = \Q(\eta), q_i,  \mathcal{B}_i, \alpha_{ij}, k_{i,j}, z_{ij}$ and $\beta_j$ be as in the above paragraph. Then $\Bigg\{1, \dfrac{\eta^j + \beta_j}{\prod\limits_{\ell=1}^{r}q_{\ell}^{k_{\ell,j}} }~\bigg\vert~1\leq j\leq n-1\Bigg\}$ is an integral basis of $L$.
\end{theorem}
The following corollary will be quickly deduced using Theorems \ref{main} and \ref{main2}. 
\begin{corollary} \label{main3}
Let $K = \Q(\theta)$ with $\theta$ a root of an irreducible polynomial $x^{9}-a$, where $a$ is a squarefree integer such that $9\nmid (a^2-1).$ Then $\{1, \theta, \theta^2, \cdots, \theta^{8}\}$ is an integral basis of $K$.
\end{corollary}
\subsection{Examples} The following examples are applications of Theorems \ref{main}, \ref{main2}. In these examples, $K = \Q(\theta)$ with $\theta$ a root of $f(x)$ and $d_K$ denotes the discriminant of $K$.  
\begin{example}
Let $f(x) = x^9 - 54$. In view of assertions (1), (2), (3) of Theorem \ref{main}, only the prime $3$ divides the index $I_K(\theta)$. Using Theorem \ref{main}(2), it can be easily checked that $v_3(I_K(\theta)) = 13$ and the set $$S= \{ 1,\theta, \theta^2, \frac{q_2(\theta)}{3}, \frac{\theta q_2(\theta)}{3}, \frac{\theta^2 q_2(\theta)}{9},\frac{q_1(\theta)}{27}, \frac{\theta q_1(\theta)}{27}, \frac{\theta^2 q_1(\theta)}{27}   \}.$$ is a $3$-integral basis of $K$ where $q_2(\theta) = \theta^3 +12\theta +12$ and $q_1(\theta)= \theta^6 +6\theta^4 + 6\theta^3 + 36\theta^2 + 72\theta+ 252$. 
By Theorem \ref{main2}, it follows that $$S= \{ 1,\theta, \theta^2, \frac{q_2(\theta)}{3}, \frac{\theta q_2(\theta)}{3}, \frac{\theta^2 q_2(\theta)}{9},\frac{q_1(\theta)}{27}, \frac{\theta q_1(\theta)}{27}, \frac{\theta^2 q_1(\theta)}{27}   \}.$$ is an integral basis of $K$.
Note that $discr(f(x)) = 9^9\cdot 54^{8}= 2^8\cdot 3^{42}$. Keeping in mind that $I_K(\theta)= 3^{13}$ and the fact that $d_K I_K(\theta)^2= discr(f(x))$, we get $d_K= 2^8\cdot 3^{16}$.
\end{example}
\begin{example}
Let $f(x) = x^{9} - 108.$ Applying assertions (1), (2), (3) of Theorem \ref{main}, we see that only the primes $2$ and $3$ divide $I_K(\theta)$. By Theorem \ref{main}(2), it follows that $v_3(I_K(\theta)) = 12$ and the set $$S= \{ 1,\theta, \theta^2, \frac{q_2(\theta)}{3}, \frac{\theta q_2(\theta)}{3}, \frac{\theta^2 q_2(\theta)}{9},\frac{q_1(\theta)}{9}, \frac{\theta q_1(\theta)}{27}, \frac{\theta^2 q_1(\theta)}{27}   \}.$$ is a $3$-integral basis of $K$ where $ q_2(\theta) = \theta^3 +24\theta +24$ and $q_1(\theta)= \theta^6 +12\theta^4 + 12 \theta^3 +144\theta^2  +288\theta + 1872$. Applying Theorem \ref{main}(1), we see that $v_2(I_K(\theta)) = 5$ and a  $2$-integral basis of $K$ is given by $$S= \{ 1,\theta, \theta^2, \theta^3,\frac{\theta^4}{2}, \frac{\theta^5}{2}, \frac{\theta^6}{2},\frac{\theta^7}{2}, \frac{\theta^8}{2}  \}.$$

As in the notations in the preceding paragraph of Theorem \ref{main2}, we see that $\delta_{1j} = 0$ for $1\leq j\leq 8$ and $\delta_{2j} = 0$ for $1\leq j\leq 2$, $\delta_{23} =24\theta+24 , \delta_{24}= 24\theta^2+24\theta, \delta_{25}= 24\theta^3+24\theta^2,\delta_{26} =12\theta^4 + 12 \theta^3 +144\theta^2  +288\theta + 1872,    
 \delta_{27} =12\theta^5 + 12 \theta^4 +144\theta^3  +288\theta^2 + 1872\theta,$ and $ \delta_{28} =12\theta^6 + 12 \theta^5 +144\theta^4  +288\theta^3 + 1872\theta^2$. Also, $k_{1,j} = 0$ for $1\leq j\leq 3$, $k_{1,j} = 1$ for $4\leq j\leq 8$, and $k_{2,j} = 0$ for $j=1,2$, $k_{2,j} = 1,$ for $j=3,4$, $k_{2,j}=2 $ for $j=5,6$, $k_{2,j}=3$ for $7\leq j\leq 8.$ So, $z_{1j} = 1$ for $j=1,2$, $z_{13} = z_{14} = 3$, $z_{1j}=9$ for $j=5,6$, $z_{1j}=27$ for $j=7,8$  and $z_{2j} = 1$ for $1\leq j\leq 3$, $z_{24}= z_{25} = z_{26} = z_{27} = z_{28} = 2.$ Hence, we can choose $u_{1j}  = 1$ for $3\leq j\leq 8$ and $u_{23} =-2, u_{24}= -1, u_{25}=u_{26}= -4 , u_{27}=u_{28}=-13$. Therefore, we see that $\beta_j = 0$ for $j=1,2$, $\beta_3= -48\theta -48 ,\beta_4 =-24\theta^2 -24\theta , \beta_5 = -96\theta^3 -96\theta^2, \beta_6 = -8(12\theta^4 + 12 \theta^3 +144\theta^2  +288\theta + 1872), \beta_7 = -26(12\theta^4 + 12 \theta^3 +144\theta^2  +288\theta + 1872)$ and $\beta_8 = -26(12\theta^4 + 12 \theta^3 +144\theta^2  +288\theta + 1872)$. So, it follows from Theorem \ref{main2} that the set $\{1, \theta, \theta^2, \frac{\theta^3+\beta_3}{3}, \frac{\theta^4+\beta_4}{6}, \frac{\theta^5 +\beta_5}{18}, \frac{\theta^6 +\beta_6\theta}{54}, \frac{\theta^7 +\beta_7}{54}, \frac{\theta^8 +\beta_8}{54}\}$ is an integral basis of $K$. Note that $discr(f(x)) =  9^9\cdot 108^{8}= 2^{16}\cdot 3^{42}$. Keeping in mind that $v_2(I_K(\theta))= 5 ,v_3(I_K(\theta))=12$ and the fact that $d_K I_K(\theta)^2= discr(f(x))$, we get $d_K= 2^{6}\cdot 3^{18}$.

\end{example}

 \section{Preliminary results}\label{prelim}
 The following proposition to be used in the sequel follows immediately from what has been
said in the first paragraph of Section 2.1.
\begin{proposition}\label{p1} Let $L=\Q(\eta)$ be an algebraic number field of degree $n$ with $\eta$ an algebraic integer and $p$ be a prime number. Let  $\{\alpha_1,\alpha_2,\cdots ,\alpha_{n-1}\}$ be $p$-integral elements of $L$ of the type  $\alpha_i=\frac{c_{i,0}+c_{i,1}\eta+\cdots +c_{i,i-1}\eta^{i-1}+\eta^i}{p^{k_i}}$ where $c_{i,j},~k_i$ are in $\Z$ with $0\leq k_i\leq k_{i+1}$ for $1\leq i\leq n-1$. Then  $\{1,\alpha_1,\cdots ,\alpha_{n-1}\}$ is a $p$-integral basis of $L$ if and only if    $v_p(I_L(\eta))=\displaystyle\sum_{i=1}^{n-1}k_i$, in which case  the integers $k_1,~\cdots,k_{n-1}$ are uniquely determined by  the prime $p$ and the element $\eta$ of $L$. Moreover there always exists a $p$-integral basis of $L$ of the above type.
\end{proposition}

Let $K = \Q(\theta)$ with $\theta$  a root of an irreducible polynomial $f(x) = x^{9} - a$, where $a$ is a $9$-th power-free integer. Let $a = \prod\limits_{i=1}^{8}a_i^i$, where $a_1, a_2, \cdots, a_{8}$ are relatively prime and squarefree integers.  With these, we now prove a simple lemma.
\begin{lemma}\label{1a}
Let $C_{k} = \prod\limits_{i=1}^{8}a_i^{\lfloor\frac{ik}{9}\rfloor}$ for $1\leq k \leq 8$, then $\frac{\theta^k}{C_k}$ is an algebraic integer.
\end{lemma}
\begin{proof}
Keeping in mind that $\theta^{9} = \prod\limits_{i=1}^{8}a_i^i$, it can be easily seen that the element $\frac{\theta^k}{C_k}$ satisfies the polynomial $x^{9} - \prod\limits_{i=1}^{8}a_i^{ik - 9\lfloor\frac{ik}{9}\rfloor}$. Hence it is an algebraic integer.
\end{proof}

Throughout $\Z_p$ denotes the ring of $p$-adic integers and $\F_p$ the field with $p$ elements. 
\subsection*{Newton polygons of first order, second order, and review of Montes' algorithm.} 
For  a prime number $p$ and $a$ in $\Z_p$, $v_p(a)$ stands for the $p$-adic valuation of $a$ defined by $v_p(p) = 1$ and $\bar{a}$ for the image of $a$ under the canonical homomorphism from $\Z_p$ onto $\F_p$. As in \cite{Jakhar}, we now state the following definitions and theorems.
\begin{definition} Let $p$ be a prime number and $g(x) = a_nx^n + \cdots + a_0$ be a polynomial over $\Z_p$ with $a_0a_n \neq 0$. To each non-zero term $a_ix^i$, we associate a point $(i, v_p(a_i))$ and form the set $P = \{(i, v_p(a_{i})) : 0\leq i\leq n, a_{i} \neq 0\}$. The Newton polygon of $g(x)$ with respect to $p$ (also called $p$-Newton polygon of $g(x)$ of first order) is the polygonal path formed by the lower edges along the convex hull of points of $P$. Note that the slopes of the edges are increasing when calculated from left to right.
\end{definition}
\begin{definition}\label{residue}
Let $g(x) = x^n + a_{n-1}x^{n-1} + \cdots + a_0$ be a polynomial over $\Z_p$ such that the $p$-Newton polygon of $g(x)$ consists of a single edge having negative slope $\lambda$, i.e., $\min\{\frac{v_p(a_{i})-v_p(a_0)}{i} : 1\leq i\leq n\} = -\frac{v_p(a_0)}{n} = \lambda.$ Let $e$ denote the smallest positive integer such that $e\lambda \in \Z$. Then we associate with $g(x)$ a polynomial $T(g)(Y) \in \F_{p}[Y]$ not divisible by $Y$ of degree $\frac{n}{e} = d$ (say) defined by  $$T(g)(Y) = Y^d + \sum\limits_{j=0}^{d-1}\overline{\bigg(\frac{a_{ej}}{p^{v_p(a_0)+ej\lambda}}\bigg)}Y^{j}.$$
To be more precise, $T(g)(Y)$ will be called the residual polynomial of $g(x)$  with respect to $p$. 
\end{definition}
\begin{example}
	Consider $g(x) = (x+5)^4-5$. One can see that the $2$-Newton polygon of $g(x)$ consists of only one edge having slope $\lambda = -\frac{1}{2}$. With notations as in the above definition, we have $e = 2, n = 4, d = 2$ and the residual polynomial of $g(x)$ with respect to $2$ is $T(g)(Y) = Y^2 + Y+ \bar{1}$ belonging to $\F_2[Y]$.
\end{example}
We now state the following weaker version of the theorem proved by Ore \cite{Ore} in a more general set up. Its proof is omitted.
\begin{theorem}\label{2.3}
Let $p$ be a prime number. Let $L = \Q(\beta)$ where $\beta$  is a root of a monic polynomial  $g(x) = a_nx^n + \cdots + a_0 \in \Z[x]$, $a_0\neq 0$ with $g(x) \equiv x^n$ mod $p$. Suppose that the $p$-Newton polygon of $g(x)$ of first order consists of a single edge with negative slope $\lambda$. If the residual polynomial $T(g)(Y) \in \F_p[Y]$ of $g(x)$   has no repeated roots, then the highest power of $p$ dividing the index of the subgroup $\Z[\beta]$ in $A_L$ equals the number of points with positive integer coordinates lying on or below the $p$-Newton polygon of $g(x)$.   
\end{theorem} 
\begin{definition}
Let $p, L = \Q(\beta), g(x), T(g)(Y)$ be as in the above theorem. If  the residual polynomial $T(g)(Y) \in \F_p[Y]$ of $g(x)$   has no repeated roots, then we say that $g(x)$ is $p$-regular. 
\end{definition}
Let $p, L = \Q(\beta), g(x), \lambda$ be as in Theorem \ref{2.3}. Let $\lambda = -\frac{\ell}{e}$ with $\gcd(\ell, e) = 1$ and $g(x)$ is not $p$-regular, i.e., the residual polynomial $T(g)(Y) \in \F_p[Y]$ of $g(x)$ has repeated roots. In this situation,  as in \cite{G-M-N}, we define the Newton polygon of second order. Let $T(g)(Y)$ is a power of  monic irreducibile polynomial $\psi(Y)$ over $\F_p$, i.e., $T(g)(Y) = \psi(Y)^{s}$ in $\F_p[Y]$, where $s \geq 2$. In this case,  we construct  a key polynomial $\phi(x)$ attached with the slope $\lambda$ such that the following hold:
\begin{itemize}
\item[(1)] $\phi(x)$ is congruent to a power of $x$ modulo $p$.
\item[(2)] The $p$-Newton polygon of $\phi(x)$ of first order is one-sided with slope $\lambda$.
\item[(3)] The residual polynomial of $\phi(x)$ with respect to $p$ is $\psi(Y)$ in $\F_p[Y]$.
\item[(4)] $\deg \phi(x) = e\deg \psi(Y).$
\end{itemize} 
As described in \cite[Section 2.2]{G-M-N}, the data $(x; \lambda, \psi(Y))$ determines a $p$-adic valuation $V$ of the field $\Q_p(x)$, which satisfies the following basic properties.
\begin{itemize}
\item[(1)] $V(x) = \ell$, where $\lambda = -\frac{\ell}{e}$  with $\gcd(\ell, e) = 1$.
\item[(2)] If $P(x) = \sum\limits_{0\leq i}b_ix^i \in \Z_p[x]$ is any polynomial, then 
\begin{equation}\label{valuation}
V(P(x)) = e\min_{0\leq i}\{v_p(b_i)+i |\lambda|\}.
\end{equation}
\end{itemize} 
We define the above valuation $V$ to the valuation of second order. We can use this valuation of second order to construct Newton polygons of second order of  any polynomial $g(x)$ belonging to $\Z_p[x]$ with respect to its $\phi$-adic expansion.

If $g(x) = \sum\limits_{i=0}^{m}a_i(x)\phi(x)^i \in \Z_p[x]$ is the $\phi$-adic expansion of $g(x)$, then the Newton polygon of $g(x)$ with respect to $V$ (also called $V$-Newton polygon of $g(x)$ of second order) is the lower convex hull of the set of points  $(i, V(a_i(x)\phi(x)^i))$ of the Euclidean plane.

Let the $V$-Newton polygon of $g(x)$ of second order has $k$-sides $E_1, \cdots, E_k$ with negative slopes $\lambda_1, \cdots, \lambda_k$. Let $\lambda_t = -\frac{\ell_t}{e_t}$ with $\gcd(\ell_t, e_t) = 1$ and $[a_t, b_t]$ denote the projection to the horizontal axis of the side of slope $\lambda_t$ for $1\leq t\leq k$. Then, there is a natural residual polynomial $T_t(Y)$ of second order attached to each side $E_t$, whose degree coincides with the degree of the side (i.e. $\frac{b_t-a_t}{e_t}$) \cite[Section 2.5]{G-M-N}. Only those integral points of the $V$-Newton polygon of $g(x)$ which lie on the side, determine a nonzero coefficient of this second order residual polynomial, and this is the only property we need in what follows. We define $f(x)$ to be $T_t$-regular when the second order residual polynomial $T_t(Y)$ attached to the side $E_t$ of the $V$-Newton polygon of $g(x)$ of second order is separable in $\F_p[Y]/\langle \psi(Y)\rangle$. We define $f(x)$ to be $V$-regular if $f(x)$ is $T_t$-regular for each $t$, $1\leq t\leq k$.\\

For computing $p$-integral basis and for the sake of completeness, we recall the following definition from the paper  \cite{G-M-N2}.
\begin{definition}\label{quo}
Let $\phi(x) \in \Z[x]$ be a monic polynomial, and let 
$$f(x) = a_n(x)\phi(x)^n + \cdots + a_1(x)\phi(x) + a_0(x),$$ with $a_i(x) \in \Z[x]$, $\deg a_i(x) < \deg \phi(x)$, be the $\phi$-adic expansion of $f(x)$. Then we define the quotients attached to this $\phi$-expansion, by definition, the different quotients $q_1(x), \cdots, q_n(x)$ that are obtained along the computation of the coefficients of the expansion:
\begin{align*}
f(x) &= \phi(x)q_1(x) + a_0(x),\\
q_1(x) &= \phi(x)q_2(x) + a_1(x),\\
\cdots &= \cdots ~\vspace*{0.5cm}~\cdots\vspace*{0.5cm}~~~\cdots~~~\cdots,\\
q_{n-1}(x) &= \phi(x)q_n(x) + a_{n-1}(x),\\
q_n(x) &= \phi(x)\cdot 0 + a_n(x) = a_n(x).
\end{align*}
\end{definition}

We now state the following weaker version of the theorems proved by Guardia, Montes and Nart in 2012 \cite[Theorem 4.18]{G-M-N} and in 2015  \cite{G-M-N2}. Its proof is omitted.
\begin{theorem}\label{index} Let $p$ be a prime number. Let $L = \Q(\beta)$ where $\beta$  is a root of a monic polynomial  $g(x) = a_nx^n + \cdots + a_0 \in \Z[x]$, $a_0\neq 0$ with $g(x) \equiv x^n$ mod $p$. Suppose that the $p$-Newton polygon of $g(x)$ of first order consists of a single edge with negative slope $\lambda$ and the residual polynomial of $g(x)$ is given by $T(g)(Y) = \psi(Y)^s \in \F_p[Y]$ for some $s\geq 2$, a linear polynomial $\psi(Y) (\neq Y) \in \F_p[Y]$. Let $\phi(x)$ be a key polynomial attached with the slope $\lambda$ and $V$ denote the corresponding second order valuation. Let the $V$-Newton polygon of $g(x)$ of second order has $k$-sides $E_1, \cdots, E_k$ with negative slopes $\lambda_1, \cdots, \lambda_k$. Let $\lambda_t = -\frac{\ell_t}{e_t}$ with $\gcd(\ell_t, e_t) = 1$ and $[a_t, b_t]$ denote the projection to the horizontal axis of the side of slope $\lambda_t$ for $1\leq t\leq k$. If $g(x)$ is $V$-regular, then the following holds:
\begin{enumerate}
\item The highest power of $p$ dividing the index $I_L(\beta)$ equals $N_1 + N_2$, where $N_1$ is the number of points with positive integer coordinates lying on or below the $p$-Newton polygon of $g(x)$ and $N_2$ is the number of points with positive integer coordinates lying on or below the $V$-Newton polygon of $g(x)$ and lying above the horizontal line passing through the last vertex  of this polygon.
\item Let $y_j$ denote the ordinate of the point of the $p$-Newton polygon of $g(x)$ of first order with abscissa $j$, and let $Y_j$ denote the ordinate  of the point of the $V$-Newton polygon of $g(x)$ of second order with abscissa $j$. Then the following set $S$ is a $p$-integral basis of $L$
$$S = \bigg\{\dfrac{\theta^{n-u}q_j(\theta)}{p^{\lfloor y_u + \frac{Y_{j} - jV(\phi(x))}{e}\rfloor}} : n-e< u\leq n, b_t-e_tf_t < j \leq b_t, 1\leq t\leq k \bigg\},$$ where $e$ is the smallest positive integer such that $e\lambda \in \Z$, $q_j(\theta)$ is the $j$-th quotient in the $\phi$-adic expansion of $f(x)$ as in Defintion \ref{quo} and $f_t = \frac{b_t-a_t}{e_t}$ for $1\leq t\leq r$.
\end{enumerate} 
\end{theorem}

We now state a basic lemma which will be used in the sequel. It is already known (cf. \cite[Problem 435]{Book}) and we omit its proof. 
\begin{lemma}\label{count}
Let $t, b$ be positive integers with $\gcd(t,b) = c$. Let $P$ denote the set of points in the plane with positive integer coordinates lying inside or on the triangle with vertices $(0,0), (t,0), (0,b)$. Then 
$$\#P = \sum\limits_{i=0}^{t-1}\big\lfloor\frac{ib}{t}\big\rfloor =  \frac{1}{2}[(t-1)(b-1) + c-1].$$
\end{lemma} 
\section{Computation of $p$-integral basis}
In this section, we first recall our notations introduced in the first paragraph of Section 2. Let $K = \Q(\theta)$ with $\theta$ a root of an irreducible polynomial $x^{9} - a$ belonging to $\Z[x]$, where  $a$ is a $9$-th power-free integer.   Whenever $3$ divides both $a$ and $v_{3}(a)$, we denote $\frac{v_{3}(a)}{3}$ by $c$; in this situation we can write $a = 3^{3c}b$ with $3 \nmid b$.

With the above notations, we now prove a theorem.
\begin{theorem}\label{as2}
 Let $K = \Q(\theta)$ be an algebraic number field where $\theta$ is a root of an irreducible polynomial $f(x) = x^{9} - a$ belonging to $\Z[x]$, where $a$ is $9$-th power free integer. Suppose that $3$ divides both $a$ and $v_{3}(a)$. Let $b$, $c$ be as above and $k$ be the smallest positive integer such that $b \equiv k\mod 9$. For $k\in \{1,8\}$, let $q_j(\theta)$ denote the polynomial $\sum\limits_{s=0}^{3-j}\binom{3}{s}(3^{c}b)^{s}(\theta^{3}-3^{c}b)^{3-j-s}$ for $1\leq j\leq 2$. Then the highest power of $3$ dividing the index $I_K(\theta)$ is given by 
  $$v_{3}(I_K(\theta)) =\left\{\begin{array}{cl}
 12c,& \mbox{if}~ k\in \{4,5\};\\
{12c+1}, & \mbox{}~\mbox{if}~ k\in \{1,2, 7, 8\}.
\end {array}\right.$$
Moreover, the following set $S$ is a $3$-integral basis of $K$, where $q_1(\theta)$, $q_2(\theta)$ for $k\in \{2,4,,5,7\}$ and $m_i$'s for $1\leq i\leq 7$ are as in Table \ref{tab:1} and Table \ref{tab:2} of Theorem \ref{main} depending on the values of $c$ and $k$.
\begin{align*}
S =& \{ 1,\theta, \frac{\theta^2}{3^{m_1}}, \frac{q_2(\theta)}{3^{m_2}}, \frac{\theta q_2(\theta)}{3^{m_3}}, \frac{\theta^2 q_2(\theta)}{3^{m_4}},\frac{q_1(\theta)}{3^{m_5}}, \frac{\theta q_1(\theta)}{3^{m_6}}, \frac{\theta^2 q_1(\theta)}{3^{m_7}}   \}.
\end{align*}

\end{theorem}
\begin{proof}
Since $3$ divides both $a$ and $v_{3}(a)$, the $3$-Newton polygon of $f(x)$ has single edge with vertices $(0, 3c)$ and $(9,  0)$ having slope $\lambda = -\frac{c}{3}$, where  $c = \frac{v_{3}(a)}{3}$ with $1\leq c \leq 2$ (see Figure \ref{fig:1}). 

\begin{figure}[h]
	\begin{center}
		
		\begin{tikzpicture}[thick, scale=0.7][domain=0:2]
			\draw[->] (-1,0) -- (8,0)
			node[below] {$x$};
			\draw[->] (0,-2) -- (0,4)
			node[left] {$y$};
			
			\draw (0,3) node[left]{$(0,3c)$} -- (7,0) node[below]{$(9,0)$} ;
			\foreach \point in {(0,3), (7,0)} {
				\fill[black] \point circle[radius=2pt];
			}
		\end{tikzpicture}
	\end{center}
	\caption{$3$-Newton polygon of $f(x)$}\label{fig:1}
\end{figure}
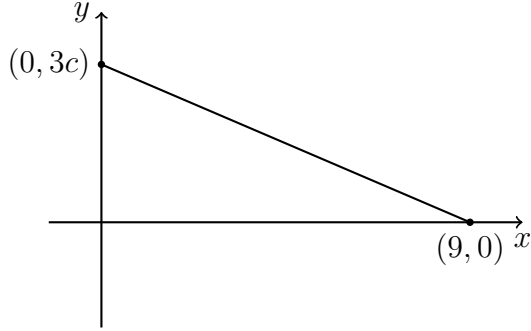

Keeping in mind that $3\nmid c$ and Definition \ref{residue}, the residual polynomial of $f(x)$ with respect to $3$ is given by $T(f)(Y) = Y^{3}-\overline{b}$ belonging to $\F_{3}(Y)$. Since $3\nmid b$, we see that $T(f)(Y) = Y^{3}-\overline{b} = (Y - \overline{b})^{3}$ = $(\psi(Y))^{3}$ (say) in $\F_{3}[Y]$. In view of Lemma \ref{count}, the number of points with positive integer coordinates lying on or below the $3$-Newton polygon of $f(x)$  is given by 
\begin{equation}\label{N1}
N_1 = \frac{(9-1)(3c-1)+2}{2} = 12c-3.
\end{equation}
Since $T(f)(Y) = (\psi(Y))^{3} $ belonging to  $\F_{3}[Y]$ has repeated roots, Theorem \ref{2.3}  is not applicable here.  So consider $\phi(x) = x^{3} - 3^{c}b$. Observe that $\phi(x)$ is a key polynomial attached with the slope $\lambda$.  As in $(\ref{valuation})$, we can define the valuation $V$ of the second order attached to the data $(x; \lambda, \psi(Y))$ such that 
\begin{equation}\label{val2}
V(x) = c, V(3) = 3 \mbox{ and }V(\phi(x)) = 3c.
\end{equation}

Keeping in mind that $a = 3^{3c}b$, the $\phi$-expansion of $f(x)$ is 
\begin{align*}
f(x) &= x^{9}-a \\
&= (x^{3})^{3}- a =  (\phi(x) + 3^{c}b)^{3}-a\\ &= \phi(x)^{3} + \binom{3}{1}\phi(x)^{2}3^{c}b + \binom{3}{2}\phi(x)3^{2c}b^2 +  3^{3c}b(b^{2} - 1).
\end{align*}
The proof of the theorem is split into two cases.
\\

\textbf{CASE 1:} $b^{2} \equiv 1$ mod $9$. In this situation, the $V$-Newton polygon of $f(x) = \phi(x)^{3} + \phi(x)^{2}3^{c+1}b + \phi(x)3^{2c+1}b^2 +  3^{3c}b(b^{2} - 1)$ is the lower convex hull of the set of points $\{(0, \geq (9c+6)), (1, 9c+3), (2, 9c+3), (3, 9c)\}$.  The $V$-Newton polygon of $f(x)$ has two edges with length of horizontal projection $1$ and $2$ having slopes $\lambda_1 (\leq -3)$ and $\lambda_2 = -\frac{3}{2}$, respectively (see Figure \ref{fig:2}).  The residual polynomials corresponding to each side of the $V$-Newton polygon of $f(x)$ are linear polynomials. Hence $f(x)$ is $V$-regular. So, applying Lemma \ref{count}, the number of points with positive integer coordinates lying on or below the $V$-Newton polygon of $f(x)$ and lying above the horizontal line passing through the last vertex $(3, 9c)$ is given by 
$$N_2 = 4.$$
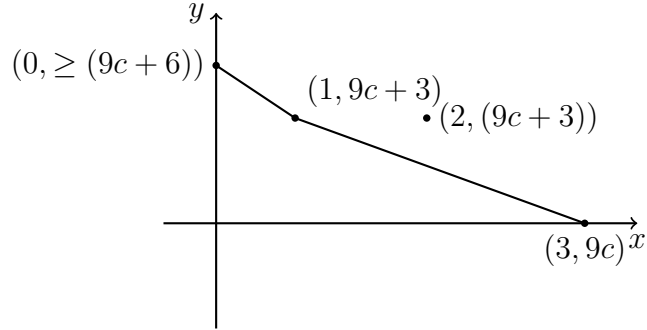
\begin{figure}[h]
	\begin{center}
		
		\begin{tikzpicture}[thick, scale=0.7][domain=0:2]
			\draw[->] (-1,0) -- (8,0)
			node[below] {$x$};
			\draw[->] (0,-2) -- (0,4)
			node[left] {$y$};
			
			\draw (0,3) node[left]{$(0,\geq (9c+6))$} -- (1.5,2) node[above right]{$(1, 9c+3)$} -- (7,0) node[below]{$(3,9c)$};
			\draw (4,2) node[right]{$(2,(9c+3))$};
			\foreach \point in {(0,3), (1.5,2), (4,2), (7,0)} {
				\fill[black] \point circle[radius=2pt];
			}
		\end{tikzpicture}
	\end{center}
	\caption{$3$-Newton polygon of $f(x)$}\label{fig:2}
\end{figure}
Therefore, by Theorem \ref{index}(1) and Equation $(\ref{N1})$, 
\begin{align*}
v_{3}(I_K(\theta)) &= N_1 + N_2 = 12c+1.
\end{align*}
As in the notations of Theorem \ref{index}(2), we see that $y_u = (3-u)\frac{c}{3}$ for $1\leq u\leq 3$, $e = 3$, $k= 2$, $a_1 = 0, b_1 = 1, e_1 = 1, f_1 = 1, a_2 = 1, b_2 = 3, e_2 = 2, f_2 = 1, Y_j = 9c + (3-j)\frac{3}{2}$ and $q_j(\theta) = \sum\limits_{s=0}^{3-j}\binom{3}{s}(3^{c}b)^{s}(\theta^{3}-3^{c}b)^{3-j-s}$ for $1\leq j\leq 2$, $q_3(\theta) = 1$. Therefore, it now follows from Theorem \ref{index}(2) that the following set $S$ is a $3$-integral basis of $K$ corresponding to the values of $c$. 
\begin{align*}
c=1:~& S= \{ 1,\theta, \theta^2, \frac{q_2(\theta)}{3}, \frac{\theta q_2(\theta)}{3}, \frac{\theta^2 q_2(\theta)}{9},\frac{q_1(\theta)}{27}, \frac{\theta q_1(\theta)}{27}, \frac{\theta^2 q_1(\theta)}{27}   \}.\\
c=2:~& S= \{ 1,\theta, \frac{\theta^2}{3}, \frac{q_2(\theta)}{9}, \frac{\theta q_2(\theta)}{27}, \frac{\theta^2 q_2(\theta)}{27},\frac{q_1(\theta)}{243}, \frac{\theta q_1(\theta)}{243}, \frac{\theta^2 q_1(\theta)}{729}   \}.
\end{align*}
\textbf{CASE 2:} $b^{2} \not\equiv 1$ mod $9$. Since $3\nmid b$, we have $3$ divides $(b^2-1)$. Hence there are four choices for $b$ as $b \equiv 2, 4, 5, 7 \mod 9$. In these cases, the choice of key polynomial $\phi(x) = x^{3} - 3^{c}b$ does not work. Hence we divide this case into two subcases corresponding to the values of $c$. \\

\noindent{\bf{subcase (i):}} Let $c=1$ and $b\equiv k\mod 9$ with $k\in\{2,4,5,7\}$. Then  consider $\phi(x) = x^3 - (-1)^{k}3bx - 3b$. It is easy to see that $\phi(x)$ is a key polynomial attached with the slope $\lambda = \frac{-1}{3}.$ As in  Equation \ref{val2}, $V(x) = 1, V(3) = 3 \mbox{ and }V(\phi(x)) = 3$.  Keeping in mind that $a = 27b$, the $\phi$-expansion of $f(x)$ is given by the following equation. 
\begin{align*}
f(x) = & \phi(x) ^3 +   \phi(x) ^2  [(-1)^k 9  b  x + 9  b] +   \phi(x)  [ 27  b ^2  x ^2 + (-1)^k(54  b ^2  x + 27  b ^3)  + 27  b ^2] +  81  b ^3  x ^2 \\ &+  81  b ^4  x + (-1)^k(81  b ^3  x + 81  b ^4)   + 27  b ^3 - 27  b.
\end{align*}
 So, one can easily check that the $V$-Newton polygon of $f(x)$ has single edge with vertices $(0, 14)$ and $(3, 9)$  or $(0, 13)$ and $(3, 9)$ according to $k\in \{2,7\}$ or not.  The residual polynomial corresponding to this single edge of the $V$-Newton polygon of $f(x)$ is linear. Hence $f(x)$ is $V$-regular. So by using Lemma \ref{count}, $N_2$ equals $4$ or $3$ according as $k\in \{2,7\}$ or not. Hence by Theorem \ref{index}(1) and $(\ref{N1})$, 
\begin{align*}
v_{3}(I_K(\theta)) &= N_1 + N_2 = 12 \text{ when }k\in\{4,5\},\\
v_{3}(I_K(\theta)) &= N_1 + N_2 = 13 \text{ when }k\in \{2,7\}.
\end{align*}
With notations as in Theorem \ref{index}(2), we see that $n = 9, e=3, a_1=0 , b_1=3, e_1=3, f_1= 1$ and $y_7= \frac{2}{3}, y_8= \frac{1}{3} , y_9 =0.$
Also, note that 
\begin{align*}
q_3(\theta) &=1, q_2(\theta) = \phi(\theta) + (-1)^k 9b\theta +9b,\\
q_1(\theta) &= \phi(\theta)^2+ \phi(\theta) ((-1)^k 9b\theta +9b) + 27b^2 \theta^2 + (-1)^k 54 b^2\theta +(-1)^k 27b^3+27b^2.
\end{align*}
When $k\in \{4,5\}$, then one can easily check that $Y_1= \frac{35}{3}, Y_2 = \frac{31}{3} $, $ Y_3 = 9$ and hence by Theorem \ref{index},  the following set is a  $3$-integral basis of $K$:  $$S= \{ 1,\theta, \theta^2, \frac{q_2(\theta)}{3}, \frac{\theta q_2(\theta)}{3}, \frac{\theta^2 q_2(\theta)}{9},\frac{q_1(\theta)}{9}, \frac{\theta q_1(\theta)}{27}, \frac{\theta^2 q_1(\theta)}{27}   \}.$$
When $k\in \{2,7\}$, then note that $Y_1= \frac{37}{3}, Y_2 = \frac{32}{3} $, $ Y_3 = 9$ and hence again using Theorem \ref{index},  the following set is a  $3$-integral basis of $K$:  $$S= \{ 1,\theta, \theta^2, \frac{q_2(\theta)}{3}, \frac{\theta q_2(\theta)}{3}, \frac{\theta^2 q_2(\theta)}{9},\frac{q_1(\theta)}{27}, \frac{\theta q_1(\theta)}{27}, \frac{\theta^2 q_1(\theta)}{27}   \}.$$

\noindent{\bf{subcase (ii):}} Let $c=2$ and $b\equiv k\mod 9$ with $k\in\{2,4,5,7\}$. Then  consider $\phi(x) = x^3 - (-1)^{\lfloor \frac{k}{2}\rfloor}3bx^2 - 9b$. It is easy to see that $\phi(x)$ is a key polynomial attached with the slope $\lambda = \frac{-2}{3}.$ As in  Equation \ref{val2}, $V(x) = 2, V(3) = 3 \mbox{ and }V(\phi(x)) = 6$.  Keeping in mind that $a = 3^{6}b$, the $\phi$-expansion of $f(x)$ is given by the following equation
\begin{align*}
f(x) = & \phi(x) ^3 +   \phi(x) ^2  [(-1)^{\lfloor \frac{k}{2}\rfloor} 9  b  x ^2 + 27  b ^2  x + (-1)^{\lfloor \frac{k}{2}\rfloor}108  b ^3   + 27  b] + \phi(x)[405  b ^4  x ^2 +\\&     (-1)^{\lfloor \frac{k}{2}\rfloor} 162  b ^2  x ^2  + (-1)^{\lfloor \frac{k}{2}\rfloor} 243  b ^5  x + 486  b ^3  x + 729  b ^6   + (-1)^{\lfloor \frac{k}{2}\rfloor} 1944  b ^4  + 243  b ^2] +\\& (-1)^{\lfloor \frac{k}{2}\rfloor} 2187  b ^7  x ^2 + 3645  b ^5  x ^2 + (-1)^{\lfloor \frac{k}{2}\rfloor} 729  b ^3  x ^2 + (-1)^{\lfloor \frac{k}{2}\rfloor} 2187  b ^6  x + 2187  b ^4  x   \\&+ 6561  b ^7  + (-1)^{\lfloor \frac{k}{2}\rfloor} 8748  b ^5   + 729  b ^3 - 729  b
\end{align*}
So, one can easily check that when $k\in \{4,5\} $, the $V$-Newton polygon of $f(x)$ has single edge with vertices  $(0, 22)$ and $(3, 18)$. When $k\in \{2,7\}$, then the $V$-Newton polygon of $f(x)$ has two edges joining the vertices $(0,23)$ to $ (1,21)$ and $ (1,21)$ to $(3,18)$.      The residual polynomial corresponding to each edge of the $V$-Newton polygon of $f(x)$ is linear. Hence $f(x)$ is $V$-regular. So by using Lemma \ref{count}, $N_2$ equals $4$ or $3$ according to $k\in \{2,7\}$ or not. Hence by Theorem \ref{index}(1) and $(\ref{N1})$, 
\begin{align*}
v_{3}(I_K(\theta)) &= N_1 + N_2 = 24 \text{ when }k\in\{4,5\}\\
v_{3}(I_K(\theta)) &= N_1 + N_2 = 25 \text{ when }k\in \{2,7\}.
\end{align*}
With notations as in Theorem \ref{index}(2), we see that $n = 9, e=3$ and $y_7= \frac{4}{3}, y_8= \frac{2}{3} , y_9 =0.$
Also, note that 
\begin{align*}
q_3(\theta) &=1, q_2(\theta) = \phi(\theta) + (-1)^{\lfloor \frac{k}{2}\rfloor} 9  b  \theta^2 + 27  b ^2  \theta + (-1)^{\lfloor \frac{k}{2}\rfloor}108  b ^3   + 27  b\\
q_1(\theta) &= \phi(\theta)^2+ \phi(\theta) (-1)^{\lfloor \frac{k}{2}\rfloor} 9  b  \theta^2 + 27  b ^2  \theta + (-1)^{\lfloor \frac{k}{2}\rfloor}108  b ^3   + 27  b) + 405  b ^4  \theta^2  \\ &+(-1)^{\lfloor \frac{k}{2}\rfloor} 162  b ^2  \theta^2  + (-1)^{\lfloor \frac{k}{2}\rfloor} 243  b ^5  \theta + 486  b ^3  \theta + 729  b ^6   + (-1)^{\lfloor \frac{k}{2}\rfloor} 1944  b ^4  + 243  b ^2
\end{align*}
When $k\in \{4,5\}$, then one can easily check that $Y_1= \frac{62}{3}, Y_2 = \frac{58}{3} $, $ Y_3 = 18$ and hence by Theorem \ref{index},  the following set is a  $3$-integral basis of $K$:  $$S= \{ 1,\theta, \frac{\theta^2}{3}, \frac{q_2(\theta)}{9}, \frac{\theta q_2(\theta)}{27}, \frac{\theta^2 q_2(\theta)}{27},\frac{q_1(\theta)}{81}, \frac{\theta q_1(\theta)}{243}, \frac{\theta^2 q_1(\theta)}{729}   \}.$$
When $k\in \{2,7\}$, then note that $Y_1= \frac{64}{3}, Y_2 = \frac{59}{3} $, $ Y_3 = 18$ and hence again using Theorem \ref{index},  the following set is a  $3$-integral basis of $K$:  
$$S= \{ 1,\theta, \frac{\theta^2}{3}, \frac{q_2(\theta)}{9}, \frac{\theta q_2(\theta)}{27}, \frac{\theta^2 q_2(\theta)}{27},\frac{q_1(\theta)}{243}, \frac{\theta q_1(\theta)}{243}, \frac{\theta^2 q_1(\theta)}{729}   \}.$$
 This completes the proof of the theorem.
\end{proof}
\begin{lemma}\label{as1}
Let $K = \Q(\theta)$ with $\theta$ satisfying an irreducible polynomial $f(x) = x^{9} - a$ having integer coefficient. Let $q$ be any prime divisor of $a$ such that $q$ does not divide $\gcd(9, v_{q}(a))$. Then $$v_q(I_K(\theta)) = \frac{8(v_q(a)-1) + \gcd(9, v_{q}(a)) - 1}{2}.$$ In this case, the set $$S = \bigg\{\frac{\theta^i}{q^{\lfloor\frac{iv_q(a)}{9}\rfloor}}~|~0\leq i\leq 8\bigg\}$$ is a $q$-integral basis of $K$.
\end{lemma}
\begin{proof}
The $q$-Newton polygon of $f(x)$ consists of only one edge having slope $\lambda = -\frac{v_q(a)}{9}$. The residual polynomial of $f(x)$ with respect to $q$ is $$T(f)(Y) = Y^{\gcd(9, v_q(a))} - \overline{a/q^{v_q(a)}} \in \F_q[Y].$$ Keeping in mind the hypothesis that $q\nmid \gcd(9, v_{q}(a))$, we see that $T(Y)$ has no repeated roots. So applying Theorem \ref{2.3} and Lemma \ref{count}, we have 
\begin{equation}\label{5.a}
v_q(I_K(\theta)) = \frac{1}{2}[8(v_q(a) - 1) + \gcd(9, v_q(a)) - 1].
\end{equation}
Now, in view of Lemma \ref{1a}, $\frac{\theta^i}{q^{\lfloor\frac{iv_q(a)}{9}\rfloor}}$ is a $q$-integral element for $1\leq i\leq 8$. Applying Lemma \ref{count}, we see that $$\sum\limits_{i=1}^{8}\lfloor\frac{iv_q(a)}{9}\rfloor = \frac{1}{2}[8(v_q(a) - 1) + \gcd(9, v_q(a)) - 1] = v_q(I_K(\theta)).$$ Therefore, in view of Proposition \ref{p1}, it follows that the set $S = \{\frac{\theta^i}{q^{\lfloor\frac{iv_q(a)}{9}\rfloor}}~|~0\leq i\leq 8\}$ is a $q$-integral basis of $K$.
\end{proof}
\begin{lemma}\label{as4}
Let $K = \Q(\theta)$ with $\theta$ satisfying an irreducible polynomial $f(x) = x^{9} - a$ having integer coefficient. Suppose that $3$ does not divide $a$, then the following hold. 
\begin{itemize}
\item[(i)] When $9 \nmid (a^{2}-1)$, then $v_{3}(I_K(\theta)) = 0$ and the set  $\{1, \theta, \cdots, \theta^{8}\}$ is a $3$-integral basis of $K$.
\item[(ii)] When $v_3(a^{2}-1) = 2$, then $v_{3}(I_K(\theta)) = 3$ and the set $\{1, \theta, \theta^2, \cdots, \theta^5, \frac{\theta^6+a\theta^3+1}{3},$ $  \frac{(\theta^6+a\theta^3+1)\theta}{3},  \frac{(\theta^6+a\theta^3+1)\theta^2}{3}\}$ is a $3$-integral basis of $K$.
\item[(iii)] When $27 \mid (a^{2}-1)$, then $v_{3}(I_K(\theta)) = 4$ and the set $\{1, \theta, \theta^2, \cdots, \theta^5, \frac{\theta^6+a\theta^3+1}{3},$ $  \frac{(\theta^6+a\theta^3+1)\theta}{3},  \frac{\theta^8+\sum\limits_{j=0}^{7}(a\theta)^{j}}{9}\}$ is a $3$-integral basis of $K$.
\end{itemize}
\end{lemma}
\begin{proof}
Set $\xi = \theta-a$. Since $\Q(\theta) = \Q(\xi)$, it is sufficient to find out $v_3(I_K(\xi))$.  Note that $\xi$ satisfies the minimal polynomial $g(x) = f(x+a) = x^{9} + 9x^8a + 9x^7a^2 + 3x^6a^3 + 9x^5a^4 + 9x^4a^5 + 3x^3a^6+9x^2a^7+9xa^8 + a(a^2-1)(a^2+1)(a^4+1)$. \\
(i) When $9 \nmid (a^{2}-1)$, then one can see that the $3$-Newton polygon of $g(x)$ has single egde having vertices $(0,1)$ and $(9,0)$. Hence the residual polynomial corresponding to this edge is linear in $\F_3[Y]$. So by Theorem \ref{2.3},  we see that $v_3(I_K(\xi))$ and hence $v_3(I_K(\theta))$ equals the number of points with positive integer coordinates lying on or below the $3$-Newton polygon of $g(x)$. Therefore, we have $v_{3}(I_K(\theta)) = 0$. Hence by Proposition \ref{p1}, it follows that $\{1, \theta, \cdots, \theta^{8}\}$ is a $3$-integral basis of $K$.\\
(ii) When $v_3(a^{2}-1) = 2$, then one can easily verify that the $3$-Newton polygon of $g(x)$ has two egdes having slopes $\frac{-1}{3}$, $\frac{-1}{6}$ and the residual polynomials corresponding to these edges are linear polynomial. Therefore  by Theorem \ref{2.3},  we see that $v_{3}(I_K(\theta)) = 3$. Now we claim that the element $\frac{\theta^6+a\theta^3+1}{3}$ is $3$-integral when $9$ divides $a^2-1$. Denote $a^2\theta^6 + a\theta^3 + 1$ by $\eta$. Keeping in mind that $\theta^9 = a$, we see that $(a\theta^3 - 1) \eta = a^4 - 1$ and hence $(a\theta^3)^3\eta^3 = \eta^3 + (a^4-1)^3 + 3\eta^2(a^4-1)+3\eta(a^4-1)^2.$ Using the fact that $\theta^9 = a$ and $a^4\neq \pm 1$ (as $f(x)$ is an irreducible polynomial), one can easily check that $\eta$ satisfies the following equation $$x^3 - 3x^2 - 3(a^4-1)x - (a^4-1)^2 = 0.$$
Hence $\frac{\eta}{3}$ satisfies the polynomial $x^3-x^2-\frac{(a^4-1)}{3}x - \frac{(a^4-1)^2}{27},$ which has integer coefficients as $9$ divides $a^2-1$. Since $9$ divides $a^2-1$, write $a^2 = 1+9k$ for some $k\in \Z$. As $\frac{a^2\theta^6 + a\theta^3 + 1}{3}$ is an algebraic integer, we have $\frac{a^2\theta^6 + a\theta^3 + 1}{3} -  3k\theta^6  = \frac{\theta^6+a\theta^3+1}{3}$ is an algebraic integer,
which proves the claim.  As $v_{3}(I_K(\theta)) = 3$, in view of Proposition \ref{p1}, it now follows that the set $\{1, \theta, \theta^2, \cdots, \theta^5, \frac{\theta^6+a\theta^3+1}{3},$ $  \frac{(\theta^6+a\theta^3+1)\theta}{3},  \frac{(\theta^6+a\theta^3+1)\theta^2}{3}\}$ is a $3$-integral basis of $K$.\\
(iii) When $v_3(a^{2}-1) \geq 3$, then one can check that the $3$-Newton polygon of $g(x)$ has three egdes having slopes $-1$, $\frac{-1}{2}$, $\frac{-1}{6}$ and the residual polynomials corresponding to these edges are linear polynomial. Therefore  by Theorem \ref{2.3},  we see that  $v_{3}(I_K(\theta)) = 4$. Now we claim that the element $\frac{\theta^8+\sum\limits_{j=0}^{7}(a\theta)^{j}}{9}$ is $3$-integral. Denote $\sum\limits_{j=0}^{8}(a\theta)^{j}$ by $\zeta$. Similar to the above case, keeping in mind that $\theta^9 = a$ and $a^{10} \neq \pm 1$, one can easily see that $\zeta$ satisfies the polynomial $x^9-\sum\limits_{j=1}^{9}\binom{9}{j}\zeta^{9-j}(a^{10}-1)^{j-1}.$ Hence $\frac{a^8\theta^8+\sum\limits_{j=0}^{7}(a\theta)^{j}}{9}$ satisfies the polynomial $$x^9-\sum\limits_{j=1}^{9}\frac{\binom{9}{j}\zeta^{9-j}(a^{10}-1)^{j-1}}{9^{j}} = g(x) \text{ (say)}.$$ Since $27$ divides $a^2 - 1$, the polynomial $g(x)$ has integer coefficients. Therefore $\frac{a^8\theta^8+\sum\limits_{j=0}^{7}(a\theta)^{j}}{9}$ is an algebraic integer. Keeping in mind the fact that $9$ divides $a^8 - 1$, we can write $a^8 = 1+9t$ for some $t\in \Z$. So $\frac{a^8\theta^8+\sum\limits_{j=0}^{7}(a\theta)^{j}}{9} - t\theta^8 = \frac{\theta^8+\sum\limits_{j=0}^{7}(a\theta)^{j}}{9}$ is an algebraic integer and hence our claim follows. Since $v_{3}(I_K(\theta)) = 4$, in view of Proposition \ref{p1} and claim proved in case (ii), it now follows that the set $\{1, \theta, \theta^2, \cdots, \theta^5, \frac{\theta^6+a\theta^3+1}{3},$ $  \frac{(\theta^6+a\theta^3+1)\theta}{3},  \frac{\theta^8+\sum\limits_{j=0}^{7}(a\theta)^{j}}{9}\}$ is a $3$-integral basis of $K$.
\end{proof}
\section{Proof of Theorem \ref{main} and Corollary \ref{main3}.}
\begin{proof}[Proof of Theorem \ref{main}.] Assertions (1), (2) and (4) of Theorem \ref{main} follow from Lemma \ref{as1}, Theorem \ref{as2} and Lemma \ref{as4}, respectively. Now we prove assertion (3). By a basic result \cite[Propositions 2.9,2.13]{Nar}, we see that
$$d_K(I_K(\theta))^2 = (-1)^{\frac{9(9-1)}{2}}N_{K/\Q}(f'(\theta)) = 9^{9}a^{8},$$ where $d_K$ is the discriminant of $K$. Hence, if a prime $q$ does not divide $3a$, then it can not divide $I_K(\theta)$, i.e., $v_q(I_K(\theta)) = 0$. Therefore, by Proposition \ref{p1}, $\{1, \theta, \theta^2, \cdots, \theta^{8}\}$ is a $q$-integral basis of $K$. This completes the proof of the theorem.
\end{proof}
\begin{proof}[Proof of Corollary \ref{main3}.] In view of assertions (1), (3) and (4)(i) of Theorem \ref{main}, we see that $\{1, \theta, \cdots, \theta^{8}\}$ is a $q$-integral basis of $K$ for any prime $q$. 
	Therefore, by Theorem \ref{main2}, $\{1, \theta, \cdots, \theta^{8}\}$ is an integral basis of $K$.
\end{proof}
\noindent{\it Conflict of interest statement, data availability statement, Author Contribution and Funding Statement:} Not Applicable.
\bibliographystyle{plain}

\end{document}